\documentclass[reqno,12pt]{amsart}
\usepackage{ragged2e}
\usepackage{amssymb}
\usepackage{amsfonts}
\usepackage{amsmath}
\usepackage{mathrsfs}
\usepackage{diagbox}
\usepackage{color,soul}
\usepackage{bbm}
\usepackage{amsthm, graphicx,empheq}
\usepackage{framed}
\usepackage[backref,colorlinks,linkcolor=red,anchorcolor=green,citecolor=blue]{hyperref}

\newcommand{\dd}{\mathrm{d}}

\newcommand{\di}{\mathrm{div}}

\newtheorem{lemma}{Lemma}[section]
\newtheorem{theorem}{Theorem}[section]
\newtheorem{definition}{Definition}[section]
\newtheorem{remark}{Remark}[section]
\numberwithin{equation}{section}

\makeatletter
\@namedef{subjclassname@2020}{\textup{2020} Mathematics Subject Classification}
\makeatother

\begin{document}
	
	\title[Uniqueness of transonic shock solutions]{Uniqueness of transonic shock solutions in general approximate nozzles for steady potential flow }
	
        \author{Minghong Han}
        \author{Bingsong Long}	
	\author{Hairong Yuan}

\address[Minghong Han]{Center for Partial Differential Equations, School of Mathematical Sciences, East China Normal University, Shanghai 200241, China}\email{\tt 52275500054@stu.ecnu.edu.cn}

\address[Bingsong Long]{School of Mathematics and Statistics, Huanggang Normal University, Hubei 438000, China}\email{\tt longbingsong@hgnu.edu.cn}

\address[Hairong Yuan]{School of Mathematical Sciences,  Key Laboratory of Mathematics and Engineering Applications (Ministry of Education) \& Shanghai Key Laboratory of PMMP,  East China Normal University, Shanghai 200241, China}\email{\tt hryuan@math.ecnu.edu.cn}



	\date{\today}

	\begin{abstract}
We study the uniqueness of solutions with a transonic shock in a two-dimensional Riemannian manifold with a special metric, which can be regarded as an approximate model of the general physical nozzles, within a class of transonic shock solutions for steady potential flow. We first prove the uniqueness of these solutions on the unit 2-sphere: for given uniform supersonic upstream flow at the entry, there exists a unique uniform pressure at the exit such that a transonic shock solution exists in the sphere, which is unique modulo a translation. A similar result is then extended to a class of manifolds. Mathematically, it is equivalent to showing a uniqueness theorem for a free boundary problem of a second-order elliptic-hyperbolic mixed-type partial differential equation in the general approximate nozzles. The proof is based on the maximum/comparison principle with a suitable special transonic shock solution as a comparison function.
	\end{abstract}

	\allowbreak
	\allowdisplaybreaks	
	\maketitle
	

\section{Introduction}\label{sec1}
It is known that the study on nozzle flows plays a fundamental role in aerodynamics and partial differential equations due to their numerous applications in practice, and closely connection with many difficult mathematical problems, such as mixed type equations and free boundary problems \cite{supshockCF,mixM}. Physically, one concerns structure of flow fields in a nozzle with transonic shocks by prescribing supersonic flow on the nozzle entry and a suitable pressure on the exit. Many efforts to analyze mathematically the stability of these special transonic shock solutions in different nozzles can be found in \cite{tssdivBF,tssCF,eutssdivC,eutssductCY,FGZ,FXeuadmi,GLY,LXY,WXX,XYY} and references cited therein. However, compared to stability, the progress in the study of uniqueness is rather slow. For potential flows, some uniqueness results on subsonic solutions can be found in \cite{Lunisubduct} and \cite{LYunisubdiv} by Liu; also, Chen--Yuan \cite{2009uniductCY} showed uniqueness of transonic shock solutions in a two-dimensional or three-dimensional straight duct under modulo a translation, and Liu--Yuan\cite{unistradivYL} extended the results to divergent nozzles. Moreover, for Euler flows, Fang--Liu--Yuan \cite{FLYunieudiv} established the uniqueness of transonic shock solutions in a straight duct with finite or infinite length. A fundamental difficulty to flow in general curved divergent nozzles is the lack of available special solutions. Motivated by a significant work of Sibners \cite{no1appSS}, Yuan, with He and Liu, had constructed various interesting special solutions \cite{exampleappY} and established their uniqueness \cite{sub2soniappYL,sub2supappYH} for subsonic-sonic flow or subsonic-supersonic flow, in two-dimensional Riemannian manifolds with “convergent–divergent” metrics, which may be regarded as approximation of de Laval nozzles.

From a mathematical perspective, a general curved divergent approximate nozzle can be described as a smooth axisymmetric two-dimensional surface in $\mathbf{R}^3$ with different Gaussian curvatures. For example, the approximate de Laval nozzles have negative Gaussian curvatures, while that of approximate straight nozzles are zero. Therefore, the study of the special flows in approximate nozzles would be of particular importance for the understanding of flows in real physical nozzles. Inspired by this idea, as a continuation of \cite{unistradivYL, LYunisubdiv, sub2soniappYL}, in this paper, we study the existence and uniqueness of solutions with a transonic shock on the unit 2-sphere with positive Gaussian curvature first, and then extend the result to general two-dimensional axisymmetric manifolds.

We now describe the problem for steady irrotational isentropic perfect gas flows on the unit 2-sphere. Let $\mathbf{S}^2=\{(\theta, \varphi)~|~\theta\in[0,\pi],~\varphi \in [0,2\pi]\}$ be the  unit 2-sphere in $\mathbf{R}^3$, with the standard metric $G=g_{ij}\dd x^i\otimes\dd x^j=\dd x^1\otimes\dd x^1+\sin^2\theta~\dd x^2\otimes\dd x^2$,  where $x^1=\theta, x^2=\varphi$. Let $u$ be a vector field on $\mathbf{S}^2$ representing the velocity of the flow. Set $\bar{u}$ as the 1-form corresponding to $u$ under the metric $G$. Then  $\bar{u}$ is exact since the flow is assumed to be irrotational and $\mathbf{S}^2$ is simply connected; namely, there exists a function $\phi$ in $\mathbf{S}^2$ called as the potential, such that $ \bar{u}= \dd \phi$. Substituting this relation into the equation of conservation of mass $\di(\rho u) = 0$ and Bernoulli's law with state function $p= \rho^{\gamma}$, we have
	\begin{align}
		&\di(\rho \nabla \phi)=0,\label{mass conservation}\\
		\frac{1}{2}|\nabla\phi|^2&+\frac{c^2}{\gamma-1}=\frac{c^2_0}{\gamma-1},\label{potential flow bernoulli}
	\end{align}
where $\gamma> 1$ and $c_0$ are two given constants; $p,\rho$ are functions on $\mathbf{S}^2$ representing respectively the scalar pressure and mass density of the flows; the sonic speed is $c=\sqrt{\gamma \rho^{ \gamma-1}}$. The operator $\nabla=(\partial_\theta,\frac{1}{\sin\theta}\partial_\varphi) $ denotes the gradient of a function on $\mathbf{S}^2$; $\di$ denotes the divergence operator on $\mathbf{S}^2$, satisfying
	\begin{align}\label{di in s2}
		\di v=\frac{1}{\sqrt{G}}\partial_i(\sqrt{G}g^{ij}v_j)=\frac{1}{\sin\theta}\big(\partial_\theta(\sin\theta v_\theta)+\partial_\varphi v_\varphi\big), 
	\end{align}
for $v=(v_\theta, v_\varphi)$, a vector field on $\mathbf{S}^2$.
	 
Applying Bernoulli's law, we get 
\begin{equation}\label{rho with phi}
    \rho=\Big(\frac{1}{\gamma}(c^2_0-\frac{1}{2}(\gamma-1)|\nabla \phi|^2)\Big)^{\frac{1}{\gamma-1}}.
\end{equation}
Therefore, the potential flow equation governing steady irrotational isentropic perfect gas flows on $\mathbf{S}^2$ reads
\begin{align}\label{potential flow in s2}
	&(c^2-\partial^2_\theta\phi)\partial_{\theta\theta}\phi+(c^2-\frac{1}{\sin^2\theta }\partial^2_\varphi\phi)\frac{1}{\sin^2\theta }\partial_{\varphi\varphi}\phi\nonumber\\
	&\quad-2\frac{1}{\sin^2\theta }\partial_{\theta}\phi\partial_{\varphi}\phi\partial_{\theta\varphi}\phi+\cot\theta(c^2+\frac{1}{\sin^2\theta }\partial^2_\varphi\phi)\partial_{\theta}\phi=0.
\end{align}
A direct computation yields that
this equation is of elliptic type if the flow is subsonic $(c^2 > (\partial_\theta\phi)^2 +(\partial_\varphi\phi)^2/\sin^2\theta)$,
and is of hyperbolic type if the flow is supersonic  $(c^2 < (\partial_\theta\phi)^2 +(\partial_\varphi\phi)^2/\sin^2\theta)$. Equivalently, by Bernoulli's law, we introduce $c_*\coloneqq\frac{2c^2_0}{\gamma+1}$, called critical sonic speed, and the flow is supersonic when $|\nabla\phi|>c_*$, while subsonic when $|\nabla\phi|<c_*$.
	
Let us specify a domain $\Omega=\{(\theta, \varphi)~|~\theta\in[\theta_0,\theta_1],~\varphi \in [0,2\pi]\}$ in $\mathbf{S}^2$, where $\theta_0>0$ and $\theta_1<\pi$ are two constants. Define $\Sigma_i=\{\theta_i\times[0,2\pi]\}~(i=0,1)$, which are respectively the entry and exit of $\Omega$; naturally $\partial\Omega=\Sigma_0\cup\Sigma_1$. Hereafter, we assume 
	\begin{align}\label{h1}
		\partial_{\theta}\phi\geq0 ~\quad \text{on}~\partial\Omega,
	\end{align}
which means exactly that the gas flows into $\Omega$ on $\Sigma_0$ and flows out of $\Omega$ on $\Sigma_1$. We are interested in the case that the flow is uniform and supersonic on $\Sigma_0$, and subsonic on $\Sigma_1$ with uniform pressure. More specifically, for a constant $u_0>c_{*}$ and a constant $0<c_{\theta_1 }<c_{*}$, we consider the following problem:
	\begin{align}\label{p1}
		\eqref{potential flow in s2} \quad &\text{in}~\Omega,\\
		\nabla\phi=(u_0,0) \quad &\text{on} ~\Sigma_0,\label{sigma0con}\\
		|\nabla \phi|=c_{\theta_1 } \quad &\text{on}~\Sigma_1.\label{p1end}
	\end{align}
We remark that the formulation of this boundary value problem is physically natural. Since the flow is supersonic near $\Sigma_0$, i.e., the equation is hyperbolic, our choice of $\phi$ in $\eqref{sigma0con}$ makes the solution neat and unique in the class of $C^2$ supersonic flow, by standard energy estimates. On the other hand, since the equation is elliptic on $\Sigma_1$, only one boundary condition is necessary. We choose the Bernoulli-type condition $\eqref{p1end}$ because assigning the pressure is of more interest from the physical point of view $($cf. \cite{supshockCF}$)$.

We study the class of piecewise smooth solution with a transonic shock for problem $\eqref{p1}$--$\eqref{p1end}$.
\begin{definition}[Transonic shock solution of potential flow on $\mathbf{S}^2$]\label{tss}
	For a $C^1$ function $\theta_s(\varphi)$ defined in $\Omega$, let		
	\begin{align*}
		S&=\{(\theta_s(\varphi),\varphi)~|~\varphi\in[0,2\pi]\},\\
		\Omega^-&=\{(\theta,\varphi)~|~\theta_0<\theta<\theta_s(\varphi),~\varphi\in[0,2\pi]\},\\
		\Omega^+&=\{(\theta,\varphi)~|~\theta_s(\varphi)<\theta<\theta_1,~\varphi\in[0,2\pi]\}.
	\end{align*}
Then $\phi\in C^{0,1}(\Omega) \cap C^2(\Omega^-\cup~\Omega^+)$ is  called a transonic shock solution of $\eqref{p1}$--$\eqref{p1end}$, if it is supersonic in $\Omega^-$ and subsonic in $\Omega^+$, satisfying Eq.$\eqref{potential flow in s2}$ in $\Omega^-\cup~\Omega^+$ and the boundary conditions $\eqref{sigma0con}$--$\eqref{p1end}$ pointwise, and also the Rankine–Hugoniot conditions on $S$:
	\begin{align}
		(\rho(|\nabla \phi_-|)\nabla \phi_--\rho(|\nabla \phi_+|)\nabla \phi_+)\cdot \mathbf{n}_s=0,\label{rh1}\\
		\frac{1}{2}|\nabla \phi_-|^2+\frac{c^2_-}{\gamma-1}=\frac{1}{2}|\nabla \phi_+|^2+\frac{c^2_+}{\gamma-1},\label{rh2}
	\end{align}
	where $\phi_+(\phi_-)$ is the right (left) limit of $\phi$ along $S$, while $c_+\, (c_-)$ is the local sonic speed corresponding to  $\phi_+\, (\phi_-)$, and $\mathbf{n}_s$ is the unit normal vector of $S$ on $\mathbf{S}^2$. The curve $S$ in $\Omega$
	is also called a shock.
	\end{definition}

We next construct a symmetric transonic shock solution to problem $\eqref{p1}$--$\eqref{p1end}$, which only depends on $\theta$. Then, the problem can be reduced to 
       \begin{align}
		(c^2-\partial^2_\theta\phi)\partial_{\theta\theta}\phi+(\cot\theta) c^2\partial_\theta \phi=0\quad &\text{in}~\Omega^-\cup ~\Omega^+,\label{phi-ode}\\
		u(\theta_0)=u_0\quad &\text{on}~\Sigma_0,\label{phi-u0}\\
		u(\theta_1)=u_1=c_{\theta_1}\quad &\text{on}~\Sigma_1,\label{phi-u1}
	\end{align}
where $u(\theta)=\partial_\theta\phi$, and Eq.\eqref{phi-ode} can be rewritten as
	\begin{align*}
		(c^2-u^2)\partial_{\theta}u+\cot\theta c^2u=0.
	\end{align*} 
Obviously, the problem is equivalent to the Cauchy problem of the following ordinary differential equation:
	\begin{align}\label{u-ode}
		\frac{\dd u}{\dd \theta }=\cot\theta \frac{c^2u}{u^2-c^2},
	\end{align}
with the initial value
	\begin{align}\label{u0}
		u(\theta_0)=u_0 \quad \text{on}~\Sigma_0,
	\end{align}
or
	\begin{align}\label{u1}
		u(\theta_1)=u_1 \quad \text{on}~\Sigma_1,
	\end{align}
respectively in $~\Omega^- $ or $~\Omega^+$.
	
\begin{remark}[Necessity of a shock]
Without loss of generality, we set $\theta_0<\pi/2<\theta_1$. On the one hand, the $C^1$ solution $u$ of $\eqref{u-ode}$--$\eqref{u0}$, which strictly increases with $\theta \in[\theta_0, \pi/2] $, satisfying $u>u_0>c_*$, exists uniquely, by uniqueness of solutions of Cauchy problems of ordinary differential equations. Similarly, the unique $C^1$ solution $u$ of $\eqref{u-ode}$ and $\eqref{u1}$ exists  in $[\pi/2, \theta_1]$, strictly increasing with $u<u_1<c_*$. On the other hand, according to $\eqref{u-ode}$, the solution $u$ cannot reach sonic speed $c$. Thus, the solution of $\eqref{u-ode}$--$\eqref{u0}$ is supersonic, while that of $\eqref{u-ode}$ and $\eqref{u1}$ is subsonic for any $\theta\in [\theta_0, \theta_1]$. Since the speed of flow cannot match at any $\theta\in[\theta_0, \theta_1]$ without a shock, a jump of $u$ $($across a shock$)$ is necessary for the solution of problem $\eqref{phi-ode}$--$\eqref{phi-u1}$.
\end{remark}
	
Let $\theta=\theta_b$ be a shock $S_b$, and $U_-(\theta_b)\coloneqq (u_-, \rho_-),U_+(\theta_b) \coloneqq (u_+, \rho_+)$ be the left and right limits of the state along $S_b$, respectively. Here, $\theta_b\in(\theta_0,\theta_1)$ is a constant to be determined. The following Rankine–Hugoniot conditions should hold on $S_b$ :
	\begin{align}
		\rho_-u_-&=\rho_+u_+,\label{backrh1}\\
		\frac{1}{2}u_-^2+\frac{c^2_-}{\gamma-1}&=\frac{1}{2}u_+^2+\frac{c^2_+}{\gamma-1}.\label{backrh2}
	\end{align}
It should be noted that the potential function $\phi$ itself is continuous across a shock. 
	\begin{lemma}\label{solvable}
		There exists a transonic shock solution $\phi_b$, only depending on $\theta$, to problem $\eqref{phi-ode}$--$\eqref{phi-u1}$ with Rankine-Hugoniot conditions $\eqref{backrh1}$--$\eqref{backrh2}$.
	\end{lemma}
\begin{proof}
By integrating \eqref{u-ode} with respect to $\theta$, we get
	\begin{align}\label{firstin}
		\frac{\ln(2c^2_0-(\gamma-1)u^2)}{1-\gamma}-\ln u=\ln \sin\theta +C,
	\end{align}
where $C$ is a constant. Inserting the initial values $\eqref{u0}$ and $\eqref{u1}$ into $\eqref{firstin}$, the solution $u(\theta)$ satisfies the following algebraic equations in $\Omega^-$ and $\Omega^+$ respectively:
	\begin{align}
	(2c^2_0-(\gamma-1)u^2)(u\sin\theta)^{\gamma-1}&=(2c^2_0-(\gamma-1)u^2_0)(u_0\sin\theta_0)^{\gamma-1}~~\quad \text{in}~\Omega^-,\label{uomega-}\\
		(2c^2_0-(\gamma-1)u^2)(u\sin\theta)^{\gamma-1}&=(2c^2_0-(\gamma-1)u^2_1)(u_1\sin\theta_1)^{\gamma-1}~~\quad \text{in}~\Omega^+.\label{uomega+}
	\end{align}
For any fixed $\theta_b\in [\theta_0,\theta_1]$, $U_-(\theta_b)$ is obtained by solving $\eqref{uomega-}$, and there exists a unique $U_+(\theta_b)$  such that the above Rankine–Hugoniot conditions hold for $(U_-(\theta_b), U_+(\theta_b))$. Furthermore, $U_+(\theta_b)$ is subsonic and the physical entropy condition $p_-(\theta_b) < p_+(\theta_b)$ holds. (The detailed calculation is similar to \cite[Proposition 3]{backremarkY} without the condition deduced from conservation of momentum equation.) Finally, for suitable $u_1$,  a transonic shock solution exists by \eqref{uomega+}.
\end{proof}

\begin{remark}
    In fact, by $\eqref{uomega+}$, $u_1$ satisfies $\eqref{u-ode}$, with $u_+$ as the initial value on $S_{b}$. Thus, in most cases, $u_1$ cannot be given in an arbitrary way; that is, the problem  $\eqref{phi-ode}$--$\eqref{phi-u1}$ is ill-posed even in the class of discontinuous solutions. 
\end{remark} 

The transonic shock solution $\phi_b$, called a background solution, will be proven to be non-unique in Subsection \ref{subsec2.1}. However, we can show that all the transonic shock solutions of $\eqref{p1}$--$\eqref{p1end}$ (as defined in Definition \ref{tss}) belonging to this class of solutions are unique modulo a translation.   
	
	\begin{theorem}[Main theorem]\label{unique}
		Under the assumption $\eqref{h1}$, for given $u_0\in(c_*, \sqrt{2c_0^2/\gamma-1})$, problem $\eqref{p1}$--$\eqref{p1end}$ is solvable for
		transonic shock solutions in the sense of Definition $\ref{tss}$ if and only if $c_{\theta_1} = u_1$, with $u_1$ being a constant in
		$(0,c_*)$ determined by $u_0, c_0$, and $ \gamma >1$. In addition, the solution is unique modulo a translation: it is exactly $\phi_b$ with $\theta_s\equiv\theta_b \in (\theta_0,\theta_1)$.
	\end{theorem}	

This result especially implies that, if the pressure is posed at the exit, then the boundary value problem is ill-posed in most cases, and the special transonic shock solutions $\phi_b$  are not physically stable. Without small perturbations of the upstream supersonic flow, our proof is global and based on the maximum/comparison principle in \cite{ellipticGT} with a suitable special transonic shock solution as a comparison function. It reveals the basic uniqueness property of such transonic shock solutions.

In the rest of this paper, we establish Theorem \ref{unique} in Section \ref{sec2}, and extend the conclusion of Theorem \ref{unique} to a class of smooth two-dimensional axisymmetric manifold in $\mathbf{R}^3$ in Section \ref{sec3}.

\section{Proof of main theorem}\label{sec2}
\subsection{The non-uniqueness of background solutions}\label{subsec2.1}
Before proving Theorem \ref{unique}, we demonstrate the non-uniqueness of background solutions $\phi_b$ obtained in Lemma \ref{solvable}. Let $u_b=\partial_\theta \phi_b$, and $u_{b+}\, (u_{b-})$ be the right (left) limit of $u_b$ along $S_b$, and $c_{b+}\, (c_{b-})$ be the sonic speed corresponding to $u_{b+}\, (u_{b-})$.

For fixed $u_0$ on $\Sigma_0$, we analyze the relation between $u_1$ and $\theta_b$. By continuous dependence of initial values for solutions of ordinary differential equations, $u_1$ is a continuously differentiable function of $u_{b+}\, (\theta_b)$. The latter is also continuously differentiable with respect to $u_-(\theta_b)$ by the Rankine–Hugoniot conditions $\eqref{backrh1}$--$ \eqref{backrh2}$. Thus, the chain rule is valid, i.e.,
	\begin{align*}
		\frac{\dd u_1(\theta_b)}{\dd \theta_b }=\frac{\dd u_1(\theta_b)}{\dd u_{b+}(\theta_b) }\cdot\frac{\dd u_{b+}(\theta_b)}{\dd u_{b-}(\theta_b) }\cdot\frac{\dd u_{b-}(\theta_b)}{\dd \theta_b }.
	\end{align*}
It follows from $\eqref{u-ode}$ that the relation between $u_{b-}$ and $\theta_b$ satisfies
	\begin{align}
		\frac{\dd u_{b-}(\theta_b)}{\dd \theta_b }=\cot\theta_b\frac{c_{b-}^2u_{b-}}{u_{b-}^2-c_{b-}^2}.\label{du-dtheta}
	\end{align}
Also, using the Rankine–Hugoniot conditions \eqref{backrh1}--\eqref{backrh2}, we obtain
	\begin{align}
	\frac{\dd u_{b+}(\theta_b)}{\dd u_{b-}(\theta_b) }=\frac{(c_{b-}^2-u_{b-}^2)u_{b-}^{\gamma-2}}{(c_{b+}^2-u_{b+}^2)u_{b+}^{\gamma-2}}<0.\label{du+du-}
	\end{align}

It remains to deal with $\frac{\dd u_1(\theta_b)}{\dd u_{b+}(\theta_b)}$. Inserting $u_{b+}$ into $\eqref{uomega+}$ yields
	\begin{align*}
		(2c^2_0-(\gamma-1)u^2_{b+})(u_{b+}\sin\theta_b)^{\gamma-1}=(2c^2_0-(\gamma-1)u^2_1)(u_1\sin\theta_1)^{\gamma-1}.
	\end{align*}
Taking the derivative with respect to $u_{b+}$ on both sides of the above equation, one has
	\begin{align}
		&(\gamma-1)\Big[-2u_{b+}(u_{b+}\sin\theta_b)^{\gamma-1}\nonumber\\
		&\quad+(2c_0^2-(\gamma-1)u_{b+}^2)(u_{b+}\sin\theta_b)^{\gamma-2}
	(\sin\theta_b+u_{b+}\cos\theta_b\frac{\dd \theta_b}{\dd u_{b+}})\Big]\nonumber\\
		=&\Big[-2(\gamma-1)u_1(u_1\sin\theta_1)^{\gamma-1}\nonumber\\
		&\qquad+(2c_0^2-(\gamma-1)u_1^2)(\gamma-1)(u_1\sin\theta_b)^{\gamma-2}\sin\theta_1 \Big]\frac{\dd u_1}{\dd u_{b+}}.\label{u+u1}
	\end{align}
Besides, according to $\eqref{du-dtheta}$--$\eqref{du+du-}$, we have
	\begin{align*}
		\frac{\dd u_{b+}(\theta_b)}{\dd \theta_b}=\frac{\cos\theta_b(-c^2_{b-}u_{b-})u^{\gamma-2}_{b-}}{\sin\theta_b(c^2_{b+}-u^2_{b+})u^{\gamma-2}_{b+}},
	\end{align*}
which implies
	\begin{align*}
		\frac{\dd \theta_b}{\dd u_{b+}}=\frac{\sin\theta_b(c^2_{b+}-u^2_{b+})u^{\gamma-2}_{b+}}{\cos\theta_b(-c^2_{b-}u_{b-})u^{\gamma-2}_{b-}}.
	\end{align*}	
After a simple calculation, the left side of $\eqref{u+u1}$ changes into 
		\begin{align}\label{*}
		2(\gamma-1)(u_{b+}\sin\theta_b)^{\gamma-2}\sin\theta_b\cdot\Big[c^2_{b+}\big(1-\frac{c^2_{b+}-u^2_{b+}}{c^2_{b-}}\cdot\frac{u^{\gamma-1}_{b+}}{u^{\gamma-1}_{b-}}\big)-u^2_{b+}\Big]=0.
	\end{align}
Here, for the identity, we have used the relation $c_{b-}^2u_{b-}^{\gamma-1}=c_{b+}^2u_{b+}^{\gamma-1}$ by the Rankine–Hugoniot condition \eqref{backrh1}. On the other hand, the algebraic expression before $\frac{\dd u_1}{\dd u_{b+}}$ is positive on the right side of $\eqref{u+u1}$. Thus, there is
	\begin{align}\label{du1dub+}
	\frac{\dd u_1(\theta_b)}{\dd u_{b+}(\theta_b) }=0.
	\end{align}

Combined with \eqref{du-dtheta}--\eqref{du+du-} and \eqref{du1dub+}, it holds
	\begin{align}
		\frac{\dd u_1(\theta_b)}{\dd \theta_b }\equiv 0, \quad\forall ~\theta_b\in [\theta_0, \theta_1],
	\end{align}
which means that $u_1$ is independent of  $\theta_b$. In other words, even if boundary conditions are given on $\Sigma_0$ and $\Sigma_1$, the shock  $S_b$ (corresponding to $\phi_b$) can translate in the  $\theta$-direction. Therefore, there is no uniqueness of $\phi_b$, for which reason we can construct a class of transonic shock solutions modulo a translation in the $ \theta$-direction of $\theta_b\in[\theta_0,\theta_1]$. 
\begin{remark}\label{keypoint}
   For the background solution, by Bernoulli's law, we find that the problem $\eqref{phi-ode}$--$\eqref{phi-u1}$ is equivalent to the following equations for the unknowns $(u_{b-},u_{b+},\theta_b, u_1):$
		\begin{align*}
		\rho_0u_0\sin\theta_0&=\rho_{b-}u_{b-}\sin\theta_b,\\
        \rho_{b-}u_{b-}\sin\theta_b&=\rho_{b+}u_{b+}\sin\theta_b,\\
			\rho_{b+}u_{b+}\sin\theta_b&=\rho_1u_1\sin\theta_1,   
		\end{align*}			
    where $u_0, \theta_0, \theta_1$ are given constants, and $\rho_{b-}, \rho_{b+}, \rho_1$ are respectively governed by \eqref{rho with phi} respect to $u_{b-}, u_{b+}, u_1$. Under the requirement $u_1<c_*$, the relation $\rho_0u_0\sin\theta_0=\rho_1u_1\sin\theta_1$ holds and yields a unique solution for $u_1$. Furthermore, for any $\theta_b\in[\theta_0,\theta_1]$, there always exists a unique pair of  $(u_{b-},u_{b+})$ satisfying the above equations with $u_{b-}>c_*>u_{b+}$ $($as the calculation in \cite[Lemma 2.1]{tssCF}$)$. Hence, $\theta_b$ is undetermined.
	\end{remark}
\begin{remark}
    The analysis of the non-uniqueness of background solutions fits the result in \cite{tssCF} that for potential flow in straight divergent nozzles with given uniform supersonic upstream flow at the entry, there exists a unique uniform pressure at the exit such that the location of the shock varies in the nozzle, even if the nozzle is not straight here.
\end{remark}
    
\subsection{Proof of Theorem \ref{unique}}\label{subsec2.2}
Let $\phi$ be a transonic shock solution of problem $\eqref{p1}$--$\eqref{p1end}$ in the sense of Definition \ref{tss}, and let $S$ be the corresponding shock given in Definition \ref{tss}. Then there is a  $\varphi_0$ such that $\theta_s(\varphi_0)=\min\{\theta_s(\varphi)~|~\varphi\in[0,2\pi]\}$. From the discussion in Subsection \ref{subsec2.1}, there is a  background transonic shock solution $\phi_b$ with corresponding shock $S_b\coloneqq\{(\theta_{b},\varphi)\}=\{(\theta_{s}(\varphi_0),\varphi)\}$. By standard energy estimates  for hyperbolic equations, there exists a unique supersonic flow solution $\phi$ satisfying $\eqref{p1}$--$\eqref{sigma0con}$ in the class of $C^2$ supersonic solutions in $\Omega^-$ (the existence follows from problem \eqref{u-ode}--\eqref{u0}). Hence, it remains to consider the solution in the subsonic region.

Let $\Omega^+ $ and $\Omega^{+}_{b}$ be the corresponding subsonic regions of $\phi$ and $\phi_b$, and then let $\Omega^{*}\coloneqq ~\Omega^+ \cap \Omega^{+}_{b}= \Omega^+$ be the common domain. Set
		\begin{align*}
		\psi=\phi_b-\phi
	\end{align*}
in $\Omega^*$. From $\eqref{mass conservation}$, $\psi$ satisfies the following linear equation:
	\begin{align}
		L\psi&=\di (\rho_b\nabla \phi_b)-\di(\rho\nabla \phi)\nonumber\\
		&=\sum_{i,j=1}^{2}a^{ij}(\phi_b)\partial_{ij}\psi +b^1\partial_1\psi +b^2\partial_2\psi=0\quad\text{in}~\Omega^*,
	\end{align}
where
	\begin{align*}
	&\partial_1=\partial_\theta, \quad \partial_2=\partial_\varphi,\\
	&a^{11}=\sin^2\theta(c_b^2-(\partial_{\theta}\phi_b)^2),\\
	&a^{12}=a^{21}=-\partial_\theta\phi_b\partial_\varphi\phi_b=0,\\
	&a^{22}=c_b^2-\frac{1}{\sin^2\theta}(\partial_\varphi\phi_b)^2=c^2_b,
	\end{align*}
with $c_b=\sqrt{c^2_0-\frac{\gamma-1}{2}(\partial_{\theta}\phi_b)^2}$ being the sonic speed corresponding to $\phi_b$, and
\begin{align}
  b^{1}= & -\left(\frac{\gamma+1}{2} \sin^{2}\theta  \partial_{\theta \theta} \phi+\frac{\gamma-1}{2} \partial_{\varphi \varphi} \phi\right) \partial_{\theta}\left(\phi+\phi_{b}\right) \nonumber\\& +\sin\theta \cos \theta \left(c_{b}^{2}+\frac{\gamma-1}{2} \partial_{\theta} \phi \partial_{\theta}\left(\phi+\phi_{b}\right)\right), \\
  b^{2}= & \frac{\gamma-1}{2} \partial_{\varphi} \phi \partial_{\theta\theta} \phi+2 \partial_{\theta} \phi \partial_{\theta\varphi} \phi+\frac{\gamma+1}{2} \frac{1}{\sin^{2}\theta} \partial_{\varphi} \phi \partial_{\varphi\varphi} \phi \nonumber\\& +\frac{\gamma-3}{2} \cot\theta \partial_{\theta} \phi \partial_{\varphi} \phi .
\end{align}
Owing to the boundedness of $\Omega$ and the assumption $\phi \in C^2$, we know that $|\nabla^2\phi|$ is bounded in $\Omega^*$. Moreover, since $\phi_b$ is strictly subsonic, the operator $L$ is uniformly elliptic in $\Omega^*$.

By the continuity of $\phi$ and $\phi_b$ on $S$, the boundary conditions for $\psi$ are
	\begin{align}
	|\nabla \phi_b|^2-|\nabla \phi|^2&=\nabla\psi\cdot(\nabla\phi_b+\nabla\phi)=u_1^2-c_{\theta_1}^2\quad\text{on}~\Sigma_1,\label{sigma1 condition}\\
    \psi&=\phi_b-\phi\quad\text{on}~S.
	\end{align}
Notice that in \eqref{sigma1 condition}, the dot represents inner product on the sphere. Before proceeding further, we claim here that
	\begin{align}\label{condition on S}
		\psi\leq0\quad\text{on}~S.
	\end{align}
For any fixed $\tilde{\varphi}\in [0,2\pi]$, by analysis in Subsection \ref{subsec2.1}, there exists a background solution $\tilde{\phi}_b$ such that the corresponding shock $\tilde{S}_b\coloneqq\{(\tilde{\theta}_b,\varphi)\}=\{(\theta_s(\tilde{\varphi}),\varphi)\}$. Recalling the assumption of $S_b$, we have $\tilde{\theta}_b=\theta_s(\tilde{\varphi})\geq \theta_{s}(\varphi_0)=\theta_b$. Besides, the uniqueness of the solution in supersonic region and continuity of the potential function across a shock imply 
		\begin{align*}		&\phi(\theta_s(\tilde{\varphi}),\tilde{\varphi})=\tilde{\phi}_b(\tilde{\theta}_b)=\phi_b(\theta_b)+\int_{\theta_b}^{\tilde{\theta}_b}\tilde{u}_b~\dd\theta,\\		&\phi_b(\theta_s(\tilde{\varphi}))=\phi_b(\theta_b)+\int_{\theta_b}^{\tilde{\theta}_b}u_b~\dd\theta.
		\end{align*}
Here $\tilde{u}_b=\partial_\theta \tilde{\phi}_b$, and ${u}_b=\partial_\theta {\phi}_b$.  
Thus, noting $\tilde{u}_b(\theta)>c_*>u_b(\theta)$ for $\theta\in(\theta_b,\tilde{\theta}_b)$, we obtain
		\begin{align}\label{intcondition on s}
			\psi|_{S}=(\phi_b-\phi)|_{\theta=\theta_s}=\int_{\theta_b}^{\tilde{\theta}_b}(u_b-\tilde{u}_b)~\dd \theta \leq0.
			\end{align}
The claim is proved.
		
		\begin{remark}\label{psie0}
				In $\eqref{intcondition on s}$, the equality holds if and only if $\theta_b=\tilde{\theta}_b$, i.e., $\tilde{\varphi}=\varphi_0$.
		\end{remark}
		
We now prove that $\psi \equiv 0$ in $\Omega^*$. By Remark \ref{psie0}, it suffices to show that $\psi$ is a constant, which can be divided into two cases.

$\mathbf{Case~1.}$ 
$u_1\geq c_{\theta_1}$. By the strong maximum principle, $m = \min_{\Omega^*} \psi$ can be achieved only on $\partial\Omega^*$ unless $\psi$ is a constant, where $\partial \Omega^*=S\cup \Sigma_1$.

(\romannumeral1) Suppose that $m$ is achieved on $\Sigma_1 $. It follows from $\eqref{sigma1 condition}$ that
\begin{align}
	(\partial_\theta \psi)\partial_\theta(\phi_b+\phi)\geq\frac{1}{\sin^2\theta}(\frac{\partial\phi}{\partial\varphi})^2\geq0,
\end{align}
which induces $\partial_\theta \psi\geq0$ by $\eqref{h1}$. This contradicts the Hopf boundary-point lemma unless $\psi$ is a constant.

(\romannumeral2) Suppose that $m$ is achieved on $S$. Let $m=\psi(\theta_s(\hat{\varphi}),\hat{\varphi})$ for some $\hat{\varphi}\in[0,2\pi]$. It follows that $\theta'_s(\varphi)|_{\varphi=\hat{\varphi}}=0$. If this deduction were false, then
for any fixed $\delta>0$, there exists $\bar{\varphi}\in (\hat{\varphi}-\delta, \hat{\varphi}+\delta)$ such that $\theta_{s}(\bar{\varphi})>\theta_s(\hat{\varphi})$. This indicates (see Remark \ref{rem24} below)
\begin{align}\label{contradiction1}
	\psi(\theta_{s}(\bar{\varphi}),\bar{\varphi})-\psi(\theta_s(\hat{\varphi}),\hat{\varphi})<0,
\end{align} 
which contradicts the definition of $m$. 

Noting that the unit outer normal vector of $\Omega^*$ on $S$ is $\mathbf{n}_s=(-1,\theta'_s(\varphi))$, which is $(-1,0)$ if $\varphi=\hat{\varphi}$, then by the Rankine–Hugoniot condition $\eqref{rh1}$ and Bernoulli's law $\eqref{potential flow bernoulli}$, as in the proof of Lemma \ref{solvable}, we can get
\begin{align}
	\partial_\theta\phi(\hat{\varphi})=\partial_\theta\phi_b(\hat{\varphi}).
\end{align}
Thus, there is
\begin{align}
	\nabla\psi(\theta_s(\hat{\varphi}),\hat{\varphi})\cdot \mathbf{n}_s|_{\varphi=\hat{\varphi}}=0,
\end{align}
which also contradicts the Hopf boundary-point lemma unless $\psi$ is a constant.

$\mathbf{Case~2. }$
$u_1 < c_{\theta_1}$. It is similar to the analysis in Case 1 that the maximum $M$ of $\psi$ in $\Omega^*$ can be achieved only on $S$ unless $\psi$ is a constant. Thanks to $\eqref{condition on S}$, $M = 0$ and we can also get $\nabla \psi(\theta_s(\varphi_0),\varphi_0)\cdot\mathbf{n}_s|_{\varphi=\varphi_0} = 0$, a contradiction to the Hopf lemma unless $\psi$ is constant.

\begin{remark} \label{rem24}
	We explain the inequality $\eqref{contradiction1}$ briefly. To this end, we calculate $\psi$ on $S$. Similar to the deducing of  \eqref{intcondition on s}, we have 	
	\begin{align*}
		\psi(\theta_{s}(\bar{\varphi}),\bar{\varphi})=\int^{\bar{\theta}_b}_{\theta_b}(u_b-\bar{u}_b)~\dd\theta,\\
		\psi(\theta_s(\hat{\varphi}),\hat{\varphi})=\int^{\hat{\theta}_b}_{\theta_b}(u_b-\hat{u}_b)~\dd\theta,
	\end{align*}
	where $\bar{\theta}_b=\theta_s(\bar{\varphi})$ and $\hat{\theta}_b=\theta_s(\hat{\varphi})$, with corresponding velocity $\bar{u}_b$ and  $\hat{u}_b$. Then, it follows from $\theta_{s}(\bar{\varphi})>\theta_s(\hat{\varphi})$ that
	\begin{align*}
		\psi(\theta_{s}(\bar{\varphi}&),\bar{\varphi})-\psi(\theta_s(\hat{\varphi}),\hat{\varphi})=\int^{\bar{\theta}_b}_{\theta_b}(u_b-\bar{u}_b)~\dd\theta-\int^{\hat{\theta}_b}_{\theta_b}(u_b-\hat{u}_b)~\dd\theta\\
		&=\int^{\hat{\theta}_b}_{\theta_b}(u_b-\bar{u}_b)~\dd\theta-\int^{\hat{\theta}_b}_{\theta_b}(u_b-\hat{u}_b)~\dd\theta +\int^{\bar{\theta}_b}_{\hat{\theta}_b}(u_b-\bar{u}_b)~\dd\theta\\
		&=\int^{\hat{\theta}_b}_{\theta_b}(\hat{u}_b-\bar{u}_b)~\dd\theta+\int^{\bar{\theta}_b}_{\hat{\theta}_b}(u_b-\bar{u}_b)~\dd\theta.
	\end{align*}
	Note that $\hat{u}_b=\bar{u}_b>c_*$ with $\theta\in(\theta_b,\hat{\theta}_b)$ and $u_b<c_*<\bar{u}_b$ with $\theta\in(\hat{\theta}_b,\bar{\theta}_b)$. As in the proof of $\eqref{condition on S}$, we get $\eqref{contradiction1}$. Furthermore, we deduce that $\psi$  strictly decreases with respect to $\theta$ on $S$.
\end{remark}

We conclude from the above discussion that $\psi\equiv 0$ in $\Omega^*$. Thus, by Remark \ref{psie0}, it must hold that $\theta_s(\varphi)\equiv \theta_b$, and $\Omega^+=\Omega_b^+$. It follows that for $c_{\theta_1}\neq u_1$, there is no solution, and for $c_{\theta_1}= u_1$, the solution $\phi$ is unique, which is exactly $\phi_b$, as known from the previous subsection. Hence, for $c_{\theta_1}= u_1$, the solution to problem $\eqref{p1}$--$\eqref{p1end}$ is unique modulo a translation in the $\theta$-direction.

The proof of Theorem \ref{unique} has been completed.

\section{Extension of  Theorem \ref{unique}}\label{sec3}
In this section, we still consider the steady potential flow on a smooth axisymmetric surface in $\mathbf{R}^3$. Let $\mathbf{S}^1$ be the standard unit circle in $\mathbf{R}^2$, and $\mathcal{M}$ be the Riemannian manifold $\{(x^1, x^2)\in [0, 1]\times \mathbf{S}^1 \}$ with a metric
$G=g_{ij}\dd x^i\otimes\dd x^j=\dd x^1\otimes\dd x^1+n(x^1)^2~\dd x^2\otimes\dd x^2$. Here $n(\cdot)$ is a positive smooth function on $[0,1]$.
Without loss of generality, by replacing $\sin\theta$ with $n(x^1)$ and $\cot\theta$ with $n^\prime(x^1)/n(x^1)$ in \eqref{potential flow in s2}, we recover the equation governing steady potential flows in $\mathcal{M}$:
\begin{align}\label{potential flow in M}
	&(c^2-\partial^2_1 \phi)\partial_{11}\phi+(c^2-\frac{1}{n(x^1)^2 }\partial^2_2\phi)\frac{1}{n(x^1)^2 }\partial_{22}\phi\nonumber\\
	&-2\frac{1}{n(x^1)^2 }\partial_{1}\phi\partial_{2}\phi\partial_{12}\phi+\frac{n'(x^1)}{n(x^1)}(c^2+\frac{1}{n(x^1)^2}\partial^2_2\phi)\partial_{1}\phi=0,
\end{align}
where $\partial_i~(i=1,2)$ denotes the partial derivatives with respect to the local coordinates $x^i$ respectively. The boundary value problem in $\mathcal{M}$ under consideration can be formulated as
	\begin{align}\label{p1*}
			\eqref{potential flow in M} \quad &\text{in}~\Omega,\\
		\nabla\phi=(u_0,0) \quad &\text{on}~\Sigma_0,\label{sigma0con*}\\
		|\nabla \phi|=c_{1 } \quad &\text{on}~\Sigma_1,\label{p1end*}
	\end{align}
with a constant $u_0>c_{*}$ and a constant $c_{1 }<c_{*}$. Corresponding to the problem on the sphere, we continue to assume the domain $\Omega= [0, 1]\times \mathbf{S}^1$, and also set entry and exit as $\Sigma_0=\{0\}\times \mathbf{S}^1$ and $\Sigma_1=\{1\}\times\mathbf{S}^1$ respectively, consisting of $\partial\Omega$, such that 
	\begin{align}\label{h1*}
		\partial_1\phi\geq0 ~\quad \text{on}\quad\partial\Omega.
	\end{align}
We are also interested in the transonic shock solution of the problem $\eqref{p1*}$--$\eqref{p1end*}$. 
	
	\begin{definition}[Transonic shock solution of potential flow in $\mathcal{M}$]\label{tss*}
		For a $C^1$ function $x^1_s(x^2)$ defined in $\Omega$, let		
		\begin{align*}
			S&=\{(x^1_s(x^2),x^2)~|~x^2\in[0,2\pi]\},\\
			\Omega^-&=\{(x^1,x^2)~|~x^1<x^1_s(x^2),~x^2\in[0,2\pi]\},\\
			\Omega^+&=\{(x^1,x^2)~|~x^1>x^1_s(x^2),~x^2\in[0,2\pi]\}.
		\end{align*}
		Then $\phi\in C^{0,1}(\Omega) \cap C^2(\Omega^-\cup~\Omega^+)$ is a transonic shock solution of $\eqref{p1*}$--$\eqref{p1end*}$ if it is supersonic in $\Omega^-$
		and subsonic in $\Omega^+$, satisfying $\eqref{potential flow in M}$ in $\Omega^-\cup~\Omega^+$ and the boundary conditions $\eqref{sigma0con*}$--$\eqref{p1end*}$ pointwise, and also the Rankine–Hugoniot conditions on $S$:
		\begin{align}
			(\rho(|\nabla \phi_-|)\nabla \phi_--\rho(|\nabla \phi_+|)\nabla \phi_+)\cdot \mathbf{n}_s=0,\label{rh1M}\\
			\frac{1}{2}|\nabla \phi_-|^2+\frac{c^2_-}{\gamma-1}=\frac{1}{2}|\nabla \phi_+|^2+\frac{c^2_+}{\gamma-1},\label{rh2M}
		\end{align}
		where $\phi_+\,(\phi_-)$ is the right (left) limit of $\phi$ along $S$, while $c_+\, (c_-)$ is local sonic speed corresponding to  $\phi_+\, (\phi_-)$, and $\mathbf{n}_s$ is the unit  normal vector of $S$ in $\mathcal{M}$. The curve $S$ in $\Omega$ is also called a shock.
	\end{definition}
	
	In the same way, we construct a background transonic shock solution $\phi_b$ depending only on $x^1$ with a shock $S_b$ at $x^1=x^1_b$ as in Section \ref{sec1}, and prove its non-uniqueness as in Subsection \ref{subsec2.1}. The key point of the proof lies in integrating the ordinary differential equation (similar to $\eqref{u-ode}$) with respect to $x^1$, which leads to the following relations:
	\begin{align}
		(2c^2_0-(\gamma-1)u^2)(un(x^1))^{\gamma-1}&=(2c^2_0-(\gamma-1)u^2_0)(u_0n(0))^{\gamma-1}~~\quad \text{in}~\Omega^-,\label{uomega-*}\\
		(2c^2_0-(\gamma-1)u^2)(un(x^1))^{\gamma-1}&=(2c^2_0-(\gamma-1)u^2_1)(u_1n(1))^{\gamma-1}~~\quad \text{in}~\Omega^+,\label{uomega+*}
	\end{align} 
	with the same structure as Eqs.$\eqref{uomega-}$--$\eqref{uomega+}$. Then, by analogy with the discussion in Remark \ref{keypoint}, the Rankine–Hugoniot conditions hold for any location of the shock $S_b$. 
	
	As a result, we can extend Theorem \ref{unique} to the case of the manifold $\mathcal{M}$.
	
	\begin{theorem}\label{unique*}
	Under the assumption $\eqref{h1*}$, for given $u_0\in(c_*, \sqrt{2c_0^2/\gamma-1})$, the problem $\eqref{p1*}$--$\eqref{p1end*}$ is solvable for transonic shock solutions in the sense of Definition $\ref{tss*}$ if and only if $c_{1} = u_1$, with $u_1$ being a constant in
	$(0,c_*)$ determined by $u_0, c_0$ and $ \gamma >1$. In addition, the solution is unique modulo a translation along $x^1$-direction: it is exactly $\phi_b$ with $x^1_s\equiv x^1_b \in [0,1]$.
\end{theorem}
The proof of Theorem \ref{unique*} is similar to that of Theorem \ref{unique} with the exchange of coordinates and the metric, so we omit the details.
   \begin{remark}
		Without loss of generality, let $n(t)$ satisfy: $(1)~n^\prime(t)>0;$ $(2)~n^{\prime\prime}(t)>0$ for $t\in[0,1]$.  Such a manifold $\mathcal{M}$ can be regarded as an approximation of a divergent two-dimensional nozzle, and the study is related to the potential flow in the divergent part of a de Laval nozzle.
    \end{remark}
    \begin{remark}
        In fact, for potential flow in this class of manifolds, the relations \eqref{uomega-}--\eqref{uomega+} combined with the Rankine-Hugoniot conditions lead to the undetermined ordinary differential equations, as noted in Remark $\ref{keypoint}$. This property results in the non-uniqueness of background solutions discussed earlier, a feature that differs from that of Euler flow in \cite{backremarkY}.
    \end{remark}
    \begin{remark}
        Theorem $\ref{unique*}$ indicates that the uniqueness of the transonic shock solution of problem $\eqref{p1*}$--$\eqref{p1end*}$ is independent of the Gaussian curvature $K=-n^{\prime\prime}(x^1)/n(x^1)$ of the two-dimensional manifold $\mathcal{M}$. 
    \end{remark}
    
	In conclusion, we prove the uniqueness property (modulo a translation) of the transonic shock solution in the sense of Definition \ref{tss*} for such a class of two-dimensional axisymmetric manifold $\mathcal{M}$ in $\mathbf{R}^3$.

\section*{Acknowledgments}
This work is supported by the Science and Technology Commission of Shanghai Municipality under Grants No. 24ZR1420000 and No. 22DZ2229014; National Natural Science Foundation of Hubei under Grant No. 2024AFB007. The third author acknowledges the support of Tianyuan Mathematics Research Center where this work is finished.


\begin{thebibliography}{99}

\bibitem{tssdivBF} Myoungjean Bae and Mikhail Feldman,
\newblock  \emph{Transonic shocks in multidimensional divergent nozzles},
\newblock Arch. Ration. Mech. Anal., 201(3):777-840, 2011.

\bibitem{tssCF} Gui-Qiang Chen and Mikhail Feldman,
\newblock  \emph{Multidimensional transonic shocks and free boundary prob- lems for nonlinear equations of mixed type},
\newblock J. Amer. Math. Soc., 16(3):461-494, 2003.

\bibitem{2009uniductCY}  Gui-Qiang Chen and Hairong Yuan,
\newblock  \emph{Uniqueness of transonic shock solutions in a duct for steady potential flow},
\newblock J. Differential Equations. 247(2):564-573, 2009.

\bibitem{eutssdivC} Shuxing Chen,
\newblock  \emph{Compressible flow and transonic shock in a diverging nozzle},
\newblock Comm. Math. Phys., 289(1):75-106, 2009.

\bibitem{eutssductCY} Shuxing Chen and Hairong Yuan,
\newblock  \emph{Transonic shocks in compressible flow passing a duct for three- dimensional Euler systems},
\newblock Arch. Ration. Mech. Anal., 187(3):523-556, 2008.

\bibitem{supshockCF} R. Courant and K. O. Friedrichs,
\newblock  \emph{Supersonic Flow  and Shock  Waves},
\newblock Interscience Publishers, Inc., New York, 1948.

\bibitem{FGZ}  Beixiang Fang, Xin Gao, and Qin Zhao,
\newblock  \emph{Asymptotic analysis of transonic shocks in divergent nozzles with respect to the expanding angle},
\newblock J. Differential Equations, 379:290-314, 2024.

\bibitem{FLYunieudiv} Beixiang Fang, Li Liu, and Hairong Yuan,
\newblock  \emph{Global uniqueness of transonic shocks in two-dimensional steady compressible Euler flows},
\newblock Arch. Ration. Mech. Anal. , 207(1):317-345, 2013.

\bibitem{FXeuadmi} Beixiang Fang and Zhouping Xin,
\newblock  \emph{On admissible locations of transonic shock fronts for steady Euler flows in an almost flat finite nozzle with prescribed receiver pressure},
\newblock Comm. Pure Appl. Math., 74(7):1493-1544, 2021.

\bibitem{GLY} Junlei Gao, Li Liu, and Hairong Yuan,
\newblock  \emph{On stability of transonic shocks for stationary Rayleigh flows in two-dimensional ducts},
\newblock SIAM J. Math. Anal., 52(5):5287-5337, 2020.

\bibitem{ellipticGT} David Gilbarg and Neil S. Trudinger,
\newblock  \emph{Elliptic partial differential equations  of second order. Classics in Mathematics},
\newblock Springer-Verlag, Berlin, 2001. Reprint of the 1998 edition.

\bibitem{LXY} Jun Li, Zhouping Xin, and Huicheng Yin,
\newblock  \emph{Transonic shocks for the full compressible Euler system in a general two-dimensional de Laval nozzle},
\newblock Arch. Ration. Mech. Anal. , 207(2):533-581, 2013.

\bibitem{Lunisubduct}  Li Liu,
\newblock  \emph{Unique subsonic compressible potential flows in three-dimensional ducts},
\newblock Discrete  Contin. Dyn. Syst., 27(1):357-368, 2010.

\bibitem{unistradivYL}  Li Liu and Hairong Yuan,
\newblock  \emph{Global uniqueness of transonic shocks in divergent nozzles for steady potential flows},
\newblock SIAM J. Math. Anal., 41(5):1816-1824, 2009.

\bibitem{LYunisubdiv} Li Liu and Hairong Yuan,
\newblock  \emph{Uniqueness of symmetric steady subsonic flows in infinitely long divergent nozzles},
\newblock Z. Angew. Math. Phys., 62(4):641-647, 2011.

\bibitem{sub2soniappYL} Pan Liu and Hairong Yuan,
\newblock  \emph{Uniqueness and instability of subsonic-sonic potential flow in a convergent approximate nozzle},
\newblock Proc. Amer. Math. Soc. , 138(5):1793-1801, 2010.

\bibitem{mixM} Cathleen Synge Morawetz,
\newblock  \emph{Mixed equations and transonic flow},
\newblock J. Hyperbolic Differ. Equ., 1(1):1-26, 2004.

\bibitem{no1appSS} L. M. Sibner and R. J. Sibner,
\newblock  \emph{Transonic flow on an axially symmetric torus},
\newblock J. Math. Anal. Appl., 72(1):362-382, 1979.

\bibitem{WXX} Shangkun  Weng,  Chunjing  Xie,  and  Zhouping  Xin,
\newblock  \emph{Structural  stability  of  the  transonic  shock problem in a divergent three-dimensional axisymmetric perturbed nozzle},
\newblock SIAM  J.  Math.  Anal., 53(1):279-308, 2021.

\bibitem{XYY} Zhouping Xin, Wei Yan,  and Huicheng Yin,
\newblock  \emph{Transonic shock problem for the Euler system in a nozzle},
\newblock Arch. Ration. Mech. Anal., 194(1):1-47, 2009.

\bibitem{exampleappY} Hairong Yuan,
\newblock  \emph{Examples of steady subsonic flows in a convergent-divergent approximate nozzle},
\newblock J. Differential Equations, 244(7):1675-1691, 2008.

\bibitem{backremarkY} Hairong Yuan,
\newblock  \emph{A remark on determination of transonic shocks in divergent nozzles for steady com- pressible Euler flows},
\newblock Nonlinear Anal. Real  World Appl., 9(2):316-325, 2008.

\bibitem{sub2supappYH} Hairong Yuan and Yue He,
\newblock  \emph{Transonic potential flows in a convergent-divergent approximate nozzle},
\newblock J. Math. Anal. Appl., 353(2):614-626, 2009.

\end{thebibliography}
\end{document}